\documentclass[reqno]{amsart}
\usepackage{amsmath,amsthm,amsfonts,amssymb}
\usepackage{hyperref,color} 

\numberwithin{equation}{section}
\theoremstyle{plain}

\newtheorem{theorem}{Theorem}

\newtheorem{example}[theorem]{Example}
\newtheorem{lemma}[theorem]{Lemma}

\newtheorem{remark}[theorem]{Remark}

\def\P{\mathbb P} \def\E{\mathbb E} 

\begin{document}

\title[Limit Theorems for Marked Hawkes Processes]{Limit Theorems for Marked Hawkes Processes
with Application to a Risk Model}

\author{DMYTRO KARABASH}
\address
{Courant Institute of Mathematical Sciences\newline
\indent New York University\newline
\indent 251 Mercer Street\newline
\indent New York, NY-10012\newline
\indent United States of America}
\email{karabash@cims.nyu.edu}

\author{LINGJIONG ZHU}
\address
{Courant Institute of Mathematical Sciences\newline
\indent New York University\newline
\indent 251 Mercer Street\newline
\indent New York, NY-10012\newline
\indent United States of America}
\email{ling@cims.nyu.edu}

\date{16 August 2014. \textit{Revised:} 25 February 2015}

\begin{abstract}
This paper focuses on limit theorems for linear Hawkes processes with random marks.
We prove a large deviation principle, which answers
the question raised by Bordenave and Torrisi. A central limit theorem
is also obtained. We conclude with an example of application in finance.
\end{abstract}

\maketitle

\section{Introduction and Main Results}

\subsection{Introduction}

Let $N$ be a simple point process on $\mathbb{R}$ and let $\mathcal{F}^{-\infty}_{t}:=\sigma(N(C),C\in\mathcal{B}(\mathbb{R}), C\subset(-\infty,t])$ be an increasing family of $\sigma$-algebras. Any nonnegative $\mathcal{F}^{-\infty}_{t}$-progressively measurable process $\lambda_{t}$ with
\begin{equation}
\mathbb{E}\left[N(a,b]|\mathcal{F}^{-\infty}_{a}\right]=\mathbb{E}\left[\int_{a}^{b}\lambda_{s}ds\big|\mathcal{F}^{-\infty}_{a}\right]
\end{equation}
a.s. for all intervals $(a,b]$ is called an $\mathcal{F}^{-\infty}_{t}$-intensity of $N$. We use the notation $N_{t}:=N(0,t]$ to denote the number of
points in the interval $(0,t]$. 

A linear Hawkes process with random marks is a simple point process with $\mathcal{F}_{t}^{\infty}$-predictable intensity
\begin{equation}
\lambda_{t}:=\nu+Z_{t},\quad Z_{t}:=\sum_{\tau_{i}<t}h(t-\tau_{i},a_{i}),\label{dynamicsmarked}
\end{equation}
where $\nu>0$, the $(\tau_{i})_{i\geq 1}$ are arrival times of the points, and the $(a_{i})_{i\geq 1}$ are i.i.d. random marks,
$a_{i}$ being independent of previous
arrival times $\tau_{j}$, $j\leq i$. We assume that $N(-\infty,0]=0$.
If one considers the stationary version of the process, one should start from time $-\infty$ in \eqref{dynamicsmarked}.

We further assume that $a_{i}$ has a common distribution $q(da)$ on a measurable space $\mathbb{X}$.
Here, $h(\cdot,\cdot):\mathbb{R}^{+}\times\mathbb{X}\rightarrow\mathbb{R}^{+}$ is integrable,
i.e. $\int_{0}^{\infty}\int_{\mathbb{X}}h(t,a)q(da)dt<\infty$. 
Let $H(a):=\int_{0}^{\infty}h(t,a)dt$ for any $a\in\mathbb{X}$. We also assume that
\begin{equation}
\int_{\mathbb{X}}H(a)q(da)<1.\label{lessthanone}
\end{equation}
Let $\mathbb{P}^{q}$ denote the probability measure for the $a_{i}$'s with the common law $q(da)$. 
Under assumption \eqref{lessthanone}, it is well known that there exists a unique stationary version of the linear
marked Hawkes process satisfying the dynamics \eqref{dynamicsmarked} and that by ergodic theorem, a law
of large numbers holds, 
\begin{equation}
\lim_{t\rightarrow\infty}\frac{N_{t}}{t}=\frac{\nu}{1-\mathbb{E}^{q}[H(a)]}.
\end{equation}
This paper is organized as the following. In Section \ref{LimitThmsUnmarked}, we will
review some results about the limit theorems for unmarked Hawkes processes. In Section \ref{MainResults},
we will introduce the main results of this paper, i.e. the central limit theorem and the large deviation principle
for linear marked Hawkes processes. The proof of the central limit theorem will be given
in Section \ref{CLTProof} and the proof of the large deviation principle will be given in
Section \ref{LDPProof}. Finally, we will discuss an application of our results to a risk model in finance
in Section \ref{RiskModel}.

\subsection{Limit Theorems for Unmarked Hawkes Processes}\label{LimitThmsUnmarked}

Most of the literature about Hawkes processes considered the unmarked case, i.e. with intensity
\begin{equation}
\lambda_{t}:=\lambda\left(\sum_{\tau<t}h(t-\tau)\right),
\end{equation}
where $h(\cdot):\mathbb{R}^{+}\rightarrow\mathbb{R}^{+}$ is integrable and $\Vert h\Vert_{L^{1}}<1$
and $\lambda(\cdot):\mathbb{R}^{+}\rightarrow\mathbb{R}^{+}$ is locally integrable.

When $\lambda(\cdot)$ is linear, the Hawkes process is said to be linear and it is nonlinear otherwise.
The stability results for both linear and nonlinear Hawkes processes are known. For the linear
case, we refer to Hawkes and Oakes \cite{HawkesII}. For the nonlinear case,
Br\'{e}maud and Massouli\'{e} \cite{Bremaud} proved the stability results for $\alpha$-Lipschitz $\lambda(\cdot)$
such that $\alpha\Vert h\Vert_{L^{1}}<1$. Karabash \cite{Karabash} obtained stability results
for certain non-Lipschitz $\lambda(\cdot)$ and discontinuous $\lambda(\cdot)$.

The limit theorems for both linear and nonlinear Hawkes processes are well studied in the literature.

For the linear Hawkes process, assume $\lambda(z)=\nu+z$, for some $\nu>0$ and $\Vert h\Vert_{L^{1}}<1$, 
it has an immigration-birth representation, see
for example Hawkes and Oakes \cite{HawkesII}. For linear Hawkes process, limit theorems are very well understood. There is the law of large numbers 
(see for instance Daley and Vere-Jones \cite{Daley}), i.e.
\begin{equation}
\frac{N_{t}}{t}\rightarrow\frac{\nu}{1-\Vert h\Vert_{L^{1}}},\quad\text{as $t\rightarrow\infty$ a.s.}
\end{equation}
Moreover, Bordenave and Torrisi \cite{Bordenave} proved a large deviation principle for $(\frac{N_{t}}{t}\in\cdot)$ with the rate function
\begin{equation}\label{LinearLDP}
I(x)=
\begin{cases}
x\log\left(\frac{x}{\nu+x\Vert h\Vert_{L^{1}}}\right)-x+x\Vert h\Vert_{L^{1}}+\nu &\text{if $x\in[0,\infty)$}
\\
+\infty &\text{otherwise}
\end{cases}.
\end{equation}
By applying the techniques of large deviations, the asymptotics of the ruin probabilities for risk processes in insurance 
were studied in Stabile and Torrisi \cite{Stabile} for the light-tailed claims and in Zhu \cite{ZhuVI} for the heavy-tailed claims.

The limit theorems have also been studied for an extension of linear Hawkes processes and Cox-Ingersoll-Ross processes in Zhu \cite{ZhuV},
which has applications in short interest rate models in finance.

Recently, Bacry et al. \cite{Bacry} proved a functional central limit theorem for the linear multivariate Hawkes process 
under certain assumptions which includes the linear Hawkes process as a special case and they proved that
\begin{equation}\label{LinearCLT}
\frac{N_{\cdot t}-\cdot\mu t}{\sqrt{t}}\rightarrow\sigma B(\cdot),\quad\text{as $t\rightarrow\infty$,}
\end{equation}
weakly on $D[0,1]$ equipped with Skorokhod topology, where
\begin{equation}
\mu=\frac{\nu}{1-\Vert h\Vert_{L^{1}}}\quad\text{and}\quad\sigma^{2}=\frac{\nu}{(1-\Vert h\Vert_{L^{1}})^{3}}.
\end{equation}
Moderate deviation principle for linear Hawkes processes is obtained in Zhu \cite{ZhuIV},
which fills in the gap between central limit theorem and large deviation principle.

For nonlinear Hawkes processes,
a central limit theorem is obtained in Zhu \cite{ZhuIII}.
In Bordenave and Torrisi \cite{Bordenave}, they raised two questions about large deviations for Hawkes processes.
One question is about large deviations for nonlinear Hawkes process and the other is about large deviations
for linear marked Hawkes processes. Recently, Zhu \cite{ZhuI} considered a special case for nonlinear Hawkes processes
when $h(\cdot)$ is exponential or sums of exponentials and proved the large deviations. 
In another paper, Zhu \cite{ZhuII} proved a process-level, i.e. level-3 large deviation principle for
nonlinear Hawkes processes for general $h(\cdot)$ and hence by contraction principle, the level-1 large deviation
principle for $(N_{t}/t\in\cdot)$. In this paper, we will prove the large deviations for linear
marked Hawkes processes and thus both questions raised in Bordenave and Torrisi \cite{Bordenave} have been
answered. The large deviation theory studies the small probability of rare events. Unlike the unmarked case, the rare
events in marked Hawkes processes can also be due to the presence of random marks. It is mixture of atypical behavior
of unmarked Hawkes processes and atypical behavior of random marks. The role of random marks in the large deviations
of Hawkes processes is what we need to understand. 

\subsection{Main Results}\label{MainResults}

Before we proceed, recall that a sequence $(P_{n})_{n\in\mathbb{N}}$ of probability measures on a topological space $X$ 
satisfies the large deviation principle (LDP) with rate function $I:X\rightarrow\mathbb{R}$ if $I$ is non-negative, 
lower semicontinuous and for any measurable set $A$, we have
\begin{equation}
-\inf_{x\in A^{o}}I(x)\leq\liminf_{n\rightarrow\infty}\frac{1}{n}\log P_{n}(A)
\leq\limsup_{n\rightarrow\infty}\frac{1}{n}\log P_{n}(A)\leq-\inf_{x\in\overline{A}}I(x).
\end{equation}
Here, $A^{o}$ is the interior of $A$ and $\overline{A}$ is its closure. 
See Dembo and Zeitouni \cite{Dembo} or Varadhan \cite{VaradhanII} for general background regarding large deviations and
their applications. 
Also Varadhan \cite{Varadhan} has an excellent survey article on this subject.

For a linear marked Hawkes process satisfying
the dynamics \eqref{dynamicsmarked}, we prove the following large deviation principle in this article.

\begin{theorem}[Large Deviation Principle]\label{LDP}
Assume the conditions \eqref{lessthanone} and $H(a)>0$ with positive probability.
Also assume that there exists some $\theta>0$, so that $\int_{\mathbb{X}}e^{\theta H(a)}q(da)<\infty$.
Then, 
$\mathbb{P}(N_{t}/t\in\cdot)$ satisfies a large deviation principle with rate function,
\begin{align}
\Lambda(x)
&:=
\begin{cases}
\inf_{\hat{q}}\left\{x\mathbb{E}^{\hat{q}}[H(a)]+\nu-x+x\log\left(\frac{x}{x\mathbb{E}^{\hat{q}}[H(a)]+\nu}\right)
+x\mathbb{E}^{\hat{q}}\left[\log\frac{d\hat{q}}{dq}\right]\right\} &\text{$x\geq 0$}
\\
+\infty &\text{$x<0$}
\end{cases}\nonumber
\\
&=
\begin{cases}
\theta_{\ast}x-\nu(x_{\ast}-1) &\text{$x\geq 0$}
\\
+\infty &\text{$x<0$}
\end{cases},\nonumber
\end{align}
where the infimum of $\hat{q}$
is taken over $\mathcal{M}(\mathbb{X})$,
the space of probability measures on $\mathbb{X}$
such that $\hat{q}$ is absolutely continuous w.r.t. $q$. 
Here, $\theta_{\ast}$ and $x_{\ast}$
satisfy the following equations
\begin{equation}
\begin{cases}
x_{\ast}=\mathbb{E}^{q}\left[e^{\theta_{\ast}+(x_{\ast}-1)H(a)}\right]
\\
\frac{x}{\nu}=x_{\ast}+\frac{x}{\nu}\mathbb{E}^{q}\left[H(a)e^{\theta_{\ast}+(x_{\ast}-1)H(a)}\right]
\end{cases}.
\end{equation}
\end{theorem}

\begin{theorem}[Central Limit Theorem]\label{CLT}
Assume $\lim_{t\rightarrow\infty}t^{1/2}\int_{t}^{\infty}\mathbb{E}^{q}[h(s,a)]ds=0$ and
that \eqref{lessthanone} holds. Then, 
\begin{equation}
\frac{N_{t}-\frac{\nu t}{1-\mathbb{E}^{q}[H(a)]}}{\sqrt{t}}
\rightarrow N\left(0,\frac{\nu(1+\text{Var}^{q}[H(a)])}{(1-\mathbb{E}^{q}[H(a)])^{3}}\right),
\end{equation}
in distribution as $t\rightarrow\infty$.
\end{theorem}

\begin{remark}
Comparing Theorem \ref{LDP} and Theorem \ref{CLT} with \eqref{LinearLDP} and \eqref{LinearCLT},
it is easy to see that our results
are consistent with the limit theorems for unmarked Hawkes process.
\end{remark}

\section{Proof of Central Limit Theorem}\label{CLTProof}

\begin{proof}[Proof of Theorem \ref{CLT}]
First, let us observe that
\begin{align}
\int_{0}^{t}\lambda_{s}ds
&=\nu t+\sum_{\tau_{i}<t}\int_{\tau_{i}}^{t}h(s-\tau_{i},a_{i})ds
\\
&=\nu t+\sum_{\tau_{i}<t}H(a_{i})-\mathcal{E}_{t},\nonumber
\end{align}
where the error term $\mathcal{E}_{t}$ is given by
\begin{equation}
\mathcal{E}_{t}:=\sum_{\tau_{i}<t}\int_{t}^{\infty}h(s-\tau_{i},a_{i})ds.
\end{equation}
Therefore,
\begin{align}
\frac{N_{t}-\int_{0}^{t}\lambda_{s}ds}{\sqrt{t}}
&=\frac{N_{t}-\nu t-\sum_{\tau_{i}<t}H(a_{i})}{\sqrt{t}}+\frac{\mathcal{E}_{t}}{\sqrt{t}}\label{rearrange}
\\
&=(1-\mathbb{E}^{q}[H(a)])\frac{N_{t}-\mu t}{\sqrt{t}}+\frac{\mathbb{E}^{q}[H(a)]N_{t}-\sum_{\tau_{i}<t}H(a_{i})}{\sqrt{t}}
+\frac{\mathcal{E}_{t}}{\sqrt{t}}\nonumber,
\end{align}
where $\mu:=\frac{\nu}{1-\mathbb{E}^{q}[H(a)]}$. Rearranging the terms in \eqref{rearrange},
we get
\begin{equation}
\frac{N_{t}-\mu t}{\sqrt{t}}
=\frac{1}{1-\mathbb{E}^{q}[H(a)]}
\left[\frac{N_{t}-\int_{0}^{t}\lambda_{s}ds}{\sqrt{t}}+\frac{\sum_{\tau_{i}<t}(H(a_{i})-\mathbb{E}^{q}[H(a)])}{\sqrt{t}}
-\frac{\mathcal{E}_{t}}{\sqrt{t}}\right].
\end{equation}
It is easy to check that $\frac{\mathcal{E}_{t}}{\sqrt{t}}\rightarrow 0$ in probability as $t\rightarrow\infty$. To see this,
first notice that $\mathbb{E}[\lambda_{t}]\leq\frac{\nu}{1-\mathbb{E}^{q}[H(a)]}$ uniformly in $t$.
Let $g(t,a):=\int_{t}^{\infty}h(s,a)ds$. We have $\mathcal{E}_{t}=\sum_{\tau_{i}<t}g(t-\tau_{i},a_{i})$
and thus
\begin{align}
\mathbb{E}[\mathcal{E}_{t}]
&=\int_{0}^{t}\int_{\mathbb{X}}g(t-s,a)q(da)\mathbb{E}[\lambda_{s}]ds
\\
&\leq\frac{\nu}{1-\mathbb{E}^{q}[H(a)]}\int_{0}^{t}\int_{\mathbb{X}}g(t-s,a)q(da)ds\nonumber
\\
&=\frac{\nu}{1-\mathbb{E}^{q}[H(a)]}\int_{0}^{t}\mathbb{E}^{q}[g(s,a)]ds.\nonumber
\end{align}
Hence, by L'H\^{o}pital's rule,
\begin{align}
\lim_{t\rightarrow\infty}\frac{1}{t^{1/2}}\int_{0}^{t}\mathbb{E}^{q}[g(s,a)]ds
&=\lim_{t\rightarrow\infty}\frac{\mathbb{E}^{q}[g(t,a)]}{\frac{1}{2}t^{-1/2}}
\\
&=\lim_{t\rightarrow\infty}2t^{1/2}\int_{t}^{\infty}\mathbb{E}^{q}[h(s,a)]ds=0.\nonumber
\end{align}
Hence, $\frac{\mathcal{E}_{t}}{\sqrt{t}}\rightarrow 0$ in probability as $t\rightarrow\infty$.

Furthermore, $M_{1}(t):=N_{t}-\int_{0}^{t}\lambda_{s}ds$ and $M_{2}(t):=\sum_{\tau_{i}<t}(H(a_{i})-\mathbb{E}^{q}[H(a)])$
are both martingales. 

Moreover, since $\int_{0}^{t}\lambda_{s}ds$ is of finite variation, the quadratic
variation of $M_{1}(t)+M_{2}(t)$ is the same as the quadratic variation of $N_{t}+M_{2}(t)$.
And notice that $N_{t}+M_{2}(t)=\sum_{\tau_{i}<t}(1+H(a_{i})-\mathbb{E}^{q}[H(a)])$ which
has quadratic variation
\begin{equation}
\sum_{\tau_{i}<t}(1+H(a_{i})-\mathbb{E}^{q}[H(a)])^{2}.
\end{equation}
By the standard law of large numbers, we have
\begin{align}
\frac{1}{t}\sum_{\tau_{i}<t}(1+H(a_{i})-\mathbb{E}^{q}[H(a)])^{2}
&=\frac{N_{t}}{t}\cdot\frac{1}{N_{t}}\sum_{\tau_{i}<t}(1+H(a_{i})-\mathbb{E}^{q}[H(a)])^{2}
\\
&\rightarrow\frac{\nu}{1-\mathbb{E}^{q}[H(a)]}\cdot\mathbb{E}^{q}\left[(1+H(a)-\mathbb{E}^{q}[H(a)])^{2}\right]\nonumber
\\
&=\frac{\nu(1+\text{Var}^{q}[H(a)])}{1-\mathbb{E}^{q}[H(a)]},\nonumber
\end{align}
a.s. as $t\rightarrow\infty$. By a standard martingale central limit theorem (see. e.g. Theorem VIII-3.11 of Jacod and Shiryaev \cite{Jacod}), 
we conclude that
\begin{equation}
\frac{N_{t}-\frac{\nu t}{1-\mathbb{E}^{q}[H(a)]}}{\sqrt{t}}
\rightarrow N\left(0,\frac{\nu(1+\text{Var}^{q}[H(a)])}{(1-\mathbb{E}^{q}[H(a)])^{3}}\right),
\end{equation}
in distribution as $t\rightarrow\infty$.
\end{proof}

\section{Proof of Large Deviation Principle}\label{LDPProof}

\subsection{Limit of a Logarithmic Moment Generating Function}

In this subsection, we prove the existence of the limit of the logarithmic moment generating function
$\lim_{t\rightarrow\infty}\frac{1}{t}\log\mathbb{E}[e^{\theta N_{t}}]$ and give a variational formula 
and a more explicit formula for this limit.

\begin{theorem}\label{logarithmic}
Assume \eqref{lessthanone} and that there exists some $\theta>0$, so that $\int_{\mathbb{X}}e^{\theta H(a)}q(da)<\infty$
and $q(da)$ has a continuous density.
The limit $\Gamma(\theta)$ of the logarithmic moment generating function is
\begin{equation}
\Gamma(\theta)=\lim_{t\rightarrow\infty}\frac{1}{t}\log\mathbb{E}[e^{\theta N_{t}}]
=
\begin{cases}
\nu(f(\theta)-1) &\text{if $\theta\in(-\infty,\theta_{c}]$}
\\
+\infty &\text{otherwise}
\end{cases},
\end{equation}
where $f(\theta)$ is the minimal solution to $x=\int_{\mathbb{X}}e^{\theta+H(a)(x-1)}q(da)$ and
\begin{equation}\label{thetastar}
\theta_{c}=-\log\int_{\mathbb{X}}H(a)e^{H(a)(x_{c}-1)}q(da)>0,
\end{equation}
where $x_{c}>1$ satisfies the equation $x\int_{\mathbb{X}}H(a)e^{H(a)(x-1)}q(da)=\int_{\mathbb{X}}e^{H(a)(x-1)}q(da)$.
\end{theorem}

Now, let us prove Theorem \ref{logarithmic}.

\begin{proof}[Proof of Theorem \ref{logarithmic}]
It is well known that a linear Hawkes process has an immigration-birth representation, see e.g. Hawkes and Oakes \cite{HawkesII}.
The immigrants (roots) arrive via a standard Poisson process with constant intensity $\nu>0$. Each
immigrant generates children according to a Galton-Watson tree.
Consider a random, rooted tree (with root, i.e. immigrant, at time $0$) associated to the Hawkes process via the Galton-Watson
interpretation. Note the root is unmarked at the start of the process so the marking goes into the expectation calculation later.
Let $K$ be the number of
children of the root node, and let $S^{(1)}_t,
S^{(2)}_t,
\ldots,S^{(K)}_t$ be the number of descendants of root's $k$-th child that were born before time $t$ (including
$k$-th child if an only
if it was born before time $t$). Let $S_t$ be the total number of children in tree before time $t$ including root node.
Then
\begin{align}
F_S(t) 	&:=\E[\exp(\theta S_t)]\\
	&= \sum_{k=0}^{\infty} \E[\exp(\theta S_t)| K=k ]\P(K=k) \\
	&= \exp(\theta) \sum_{k=0}^{\infty} \P(K=k) \prod_{i=1}^k \E\left[\exp\left(\theta S_t^{(i)}\right)\right]\\
	&= \exp(\theta) \sum_{k=0}^{\infty} \E\left[\exp\left(\theta S_{t}^{(1)}\right)\right]^k \P(K=k)\\
	&= \exp(\theta) \sum_{k=0}^{\infty} \int_{\mathbb{X}} \left[ \left( \int_{0}^{t} \frac{h(s,a)}{H(a)}F_S(t-s)ds  \right)^k 
e^{-H(a)}\frac{H(a)^k}{k!} \right] q(da)\\
	&= \int_{\mathbb{X}} \exp  \left( \theta+\int_0^t h(s,a)(F_S(t-s)-1)ds \right) q(da).
\end{align}
Now observe that by definition $F_S(t)=\E[\exp(\theta S_t)]$. When $\theta\leq 0$, $F_{S}(t)$ is decreasing in $t$
and also $0\leq F_{S}(t)\leq 1$. Thus, $F_{S}(t)$ converges to a finite limit $x_{\ast}$ as $t\rightarrow\infty$, which
satisfies
\begin{equation} \label{eq:x_*}
x=\int_{\mathbb{X}} \exp \left[ \theta+H(a)(x-1) \right] q(da).
\end{equation}
Similarly, for $\theta>0$, $F_{S}(t)$ is increasing in $t$ and either $F_{S}(t)\rightarrow\infty$ as $t\rightarrow\infty$
or it converges to some finite limit $x_{\ast}$ that satisfies \eqref{eq:x_*}.

Next, we need to determine for what values of $\theta$ the solution of \eqref{eq:x_*} exists.
Let
\begin{equation}
G(x)=e^{\theta}\int_{\mathbb{X}}e^{H(a)(x-1)}q(da)-x.
\end{equation}
If $\theta=0$, then $G(x)=\int_{\mathbb{X}}e^{H(a)(x-1)}q(da)-x$ satisfies $G(1)=0$, $G(\infty)=\infty$
and $G'(1)=\mathbb{E}^{q}[H(a)]-1<0$ which implies $\min_{x>1}G(x)<0$.
Hence, there exists some critical $\theta_{c}>0$ such that $\min_{x>1}G(x)=0$.
The critical values $x_{c}$ and $\theta_{c}$ satisfy $G(x_{c})=G'(x_{c})=0$, which implies
\begin{equation}
\theta_{c}=-\log\int_{\mathbb{X}}H(a)e^{H(a)(x_{c}-1)}q(da),
\end{equation}
where $x_{c}>1$ satisfies the equation $x\int_{\mathbb{X}}H(a)e^{H(a)(x-1)}q(da)=\int_{\mathbb{X}}e^{H(a)(x-1)}q(da)$. 
Hence \eqref{eq:x_*} has finite solutions if and only if $\theta\leq\theta_{c}$. Moreover, it is easy to
see that $G(x)$ is strictly convex in $x$ and hence there can be at most two (positive) solutions to \eqref{eq:x_*}.
When $\theta<0$,
\begin{equation}
G(e^{\theta})=e^{\theta}\left[\int_{\mathbb{X}}e^{H(a)(e^{\theta}-1)}q(da)-1\right]
=e^{\theta}\left[\mathbb{E}^{q}\left[e^{H(a)(e^{\theta}-1)}\right]-1\right]<0,
\end{equation}
and $F_{S}(0)=e^{\theta}$ and $F_{S}(t)$ is decreasing in $t$ and therefore it converges to the smaller solution 
of \eqref{eq:x_*}. Similary, when $\theta>0$,
\begin{equation}
G(e^{\theta})=e^{\theta}\left[\int_{\mathbb{X}}e^{H(a)(e^{\theta}-1)}q(da)-1\right]
=e^{\theta}\left[\mathbb{E}^{q}\left[e^{H(a)(e^{\theta}-1)}\right]-1\right]>0,
\end{equation}
and $F_{S}(0)=e^{\theta}$ and $F_{S}(t)$ is increasing in $t$ and for $\theta\leq\theta_{c}$ we know that $F_{S}(t)$
converges to a finite limit and therefore it must converge to the smaller solution of \eqref{eq:x_*} as $t\rightarrow\infty$.

Finally, since random roots arrive according to a Poisson process with constant intensity $\nu>0$,
we have
\begin{equation}
F_N(t) 	:=\E[\exp(\theta N_t)]= \exp  \left[ \nu \int_0^t (F_S(t-s)-1) ds  \right].
\end{equation}
But since $F_S(s) \uparrow x_*$ as $s \to \infty$ we obtain the main result 
\begin{equation}
\frac 1t \log F_N(t)=\nu \frac 1t \left[ \int_0^t \left( F_S(s)-1 \right) ds  \right] \mathop{\longrightarrow}_{t \to \infty}
\nu(x_*-1) ,
\end{equation}
if $\theta\leq\theta_{c}$, which proves the desired formula.
\end{proof}

\subsection{Large Deviation Principle}

In this section, we prove the main result, i.e. Theorem \ref{LDP} by using the G\"{a}rtner-Ellis theorem for the upper bound
and tilting method for the lower bound. Before we proceed, let us first state a lemma that will be used in the proof
the lower bound in Theorem \ref{LDP}. This result can be found in Br\'{e}maud et al. \cite{BremaudII}.

\begin{lemma}\label{formula}[Br\'{e}maud et al. \cite{BremaudII}]
Consider a linear marked Hawkes process with intensity
\begin{equation}
\lambda_{t}:=\alpha+\beta Z_{t}:=\alpha+\beta\sum_{\tau_{i}<t}h(t-\tau_{i},a_{i}),
\end{equation}
and $\beta\mathbb{E}^{q}[H(a)]<1$, where the $a_{i}$ are i.i.d. random marks with the common law $q(da)$
independent of the previous arrival times, then
there exists a unique invariant measure $\pi$ for $Z_{t}$ such that
\begin{equation}
\int_{0}^{\infty}\lambda(z)\pi(dz)=\frac{\alpha}{1-\beta\mathbb{E}^{q}[H(a)]}.
\end{equation}
\end{lemma}

Now we are ready to prove Theorem \ref{LDP}.

\begin{proof}[Proof of Theorem \ref{LDP}]
For the upper bound, since we have Theorem \ref{logarithmic}, we can simply apply G\"{a}rtner-Ellis theorem 
(see e.g. Dembo and Zeitouni \cite{Dembo}). 
To prove the lower bound, it suffices to show that for any $x>0$, $\epsilon>0$, we have
\begin{equation}
\liminf_{t\rightarrow\infty}\frac{1}{t}\log\mathbb{P}\left(\frac{N_{t}}{t}\in B_{\epsilon}(x)\right)
\geq-\sup_{\theta\in\mathbb{R}}\{\theta x-\Gamma(\theta)\},
\end{equation}
where $B_{\epsilon}(x)$ denotes the open ball centered at $x$ with radius $\epsilon$. 

The intensity at time $t$ is $\lambda_{t}:=\lambda(Z_{t})$ where $\lambda(z)=\nu+z$
and $Z_{t}=\sum_{\tau_{i}<t}h(t-\tau_{i},a_{i})$.
We tilt $\lambda$ to $\hat{\lambda}$ and $q$ to $\hat{q}$ such that by Girsanov formula
the tilted probability measure $\hat{\mathbb{P}}$ is given by
\begin{equation}
\frac{d\hat{\mathbb{P}}}{d\mathbb{P}}\bigg|_{\mathcal{F}_{t}}
=\exp\left\{\int_{0}^{t}(\lambda(Z_{s})-\hat{\lambda}(Z_{s}))ds+\int_{0}^{t}\left[\log\left(\frac{\hat{\lambda}(Z_{s})}{
\lambda(Z_{s})}\right)
+\log\left(\frac{d\hat{q}}{dq}\right)\right]dN_{s}\right\}.
\end{equation}
Let $\mathcal{Q}_{e}$ be the set of $(\hat{\lambda},\hat{q},\hat{\pi})$ such that the marked Hawkes process with 
intensity $\hat{\lambda}(Z_{t})$ and random 
marks distributed as $\hat{q}$ is ergodic with $\hat{\pi}$ as the invariant measure of $Z_{t}$.

By Jensen's inequality, 
\begin{align}
&\frac{1}{t}\log\mathbb{P}\left(\frac{N_{t}}{t}\in B_{\epsilon}(x)\right)
\\
&=\frac{1}{t}\log\int_{\frac{N_{t}}{t}\in B_{\epsilon}(x)}\frac{d\mathbb{P}}{d\hat{\mathbb{P}}}d\hat{\mathbb{P}}\nonumber
\\
&=\frac{1}{t}\log\hat{\mathbb{P}}\left(\frac{N_{t}}{t}\in B_{\epsilon}(x)\right)
-\frac{1}{t}\log\left[\frac{1}{\hat{\mathbb{P}}\left(\frac{N_{t}}{t}\in B_{\epsilon}(x)\right)}
\int_{\frac{N_{t}}{t}\in B_{\epsilon}(x)}\frac{d\hat{\mathbb{P}}}{d\mathbb{P}}d\hat{\mathbb{P}}\right]\nonumber
\\
&\geq\frac{1}{t}\log\hat{\mathbb{P}}\left(\frac{N_{t}}{t}\in B_{\epsilon}(x)\right)
-\frac{1}{\hat{\mathbb{P}}\left(\frac{N_{t}}{t}\in B_{\epsilon}(x)\right)}
\cdot\frac{1}{t}\cdot\hat{\mathbb{E}}\left[1_{\frac{N_{t}}{t}\in B_{\epsilon}(x)}\log\frac{d\hat{\mathbb{P}}}{d\mathbb{P}}\right].\nonumber
\end{align}
By the ergodic theorem (see e.g. Chapter 16.4. of Koralov and Sinai \cite{Koralov}), 
\begin{equation}
\liminf_{t\rightarrow\infty}\frac{1}{t}\log\mathbb{P}\left(\frac{N_{t}}{t}\in B_{\epsilon}(x)\right)
\geq-\mathop{\inf_{0<K<\mathbb{E}^{\hat{q}}[H(a)]^{-1}}}_{ (\hat{\lambda},\hat{q},\hat{\pi})\in\mathcal{Q}_{e}^{x},
\hat{\lambda}=K\lambda} 
\mathcal{H}(\hat{\lambda},\hat{q},\hat{\pi}).
\end{equation}
where $\mathcal{Q}_{e}^{x}$ is defined by
\begin{equation}
\mathcal{Q}_{e}^{x}=\left\{(\hat{\lambda},\hat{q},\hat{\pi})\in\mathcal{Q}_{e}:\int\hat{\lambda}(z)\hat{\pi}(dz)=x\right\}.
\end{equation}
and the relative entropy $\mathcal{H}$ is
\begin{equation}
\mathcal{H}(\hat{\lambda},\hat{q},\hat{\pi})=\int(\lambda-\hat{\lambda})\hat{\pi}+\int\log(\hat{\lambda}/\lambda)\hat{\lambda}\hat{\pi}
+\iint\log(d\hat{q}/dq)\hat{q}\hat{\lambda}\hat{\pi}.
\end{equation}
By Lemma \ref{formula}, 
\begin{align}
&\inf_{0<K<\mathbb{E}^{\hat{q}}[H(a)]^{-1},x=\frac{\nu K}{1-K\mathbb{E}^{\hat{q}}[H(a)]},(\hat{\lambda},\hat{q},\hat{\pi})\in\mathcal{Q}_{e},\hat{\lambda}=K\lambda}
\mathcal{H}(\hat{\lambda},\hat{q},\hat{\pi})
\\
&=\inf_{K=\frac{x}{x\mathbb{E}^{\hat{q}}[H(a)]+\nu},(\hat{\lambda},\hat{q},\hat{\pi})\in\mathcal{Q}_{e},\hat{\lambda}=K\lambda}
\left\{\frac{1}{K}-1+\log K
+\mathbb{E}^{\hat{q}}\left[\log\frac{d\hat{q}}{dq}\right]\right\}\int\hat{\lambda}\hat{\pi}\nonumber
\\
&=\inf_{\hat{q}}\left\{\mathbb{E}^{\hat{q}}[H(a)]+\frac{\nu}{x}-1+\log\left(\frac{x}{x\mathbb{E}^{\hat{q}}[H(a)]+\nu}\right)
+\mathbb{E}^{\hat{q}}\left[\log\frac{d\hat{q}}{dq}\right]\right\}x\nonumber
\\
&=\inf_{\hat{q}}\left\{x\mathbb{E}^{\hat{q}}[H(a)]+\nu-x+x\log\left(\frac{x}{x\mathbb{E}^{\hat{q}}[H(a)]+\nu}\right)
+x\mathbb{E}^{\hat{q}}\left[\log\frac{d\hat{q}}{dq}\right]\right\}.\nonumber
\end{align}
Next, let us find a more explict form for the Legendre-Fenchel transform of $\Gamma(\theta)$.
\begin{equation}\label{theta}
\sup_{\theta\in\mathbb{R}}\{\theta x-\Gamma(\theta)\}=\sup_{\theta\in\mathbb{R}}\{\theta x-\nu(f(\theta)-1)\},
\end{equation}
where $f(\theta)=\mathbb{E}^{q}[e^{\theta+(f(\theta)-1)H(a)}]$. Here,
\begin{equation}
f'(\theta)=\mathbb{E}^{q}\left[(1+f'(\theta)H(a))e^{\theta+(f(\theta)-1)H(a)}\right].
\end{equation}
So the optimal $\theta_{\ast}$ for \eqref{theta} would satisfy $f'(\theta_{\ast})=\frac{x}{\nu}$ and $\theta_{\ast}$ and $x_{\ast}=f(\theta_{\ast})$
satisfy the following equations
\begin{equation}
\begin{cases}
x_{\ast}=\mathbb{E}^{q}\left[e^{\theta_{\ast}+(x_{\ast}-1)H(a)}\right]
\\
\frac{x}{\nu}=x_{\ast}+\frac{x}{\nu}\mathbb{E}^{q}\left[H(a)e^{\theta_{\ast}+(x_{\ast}-1)H(a)}\right]
\end{cases},
\end{equation}
and $\sup_{\theta\in\mathbb{R}}\{\theta x-\Gamma(\theta)\}=\theta_{\ast}x-\nu(x_{\ast}-1)$. 

On the other hand, letting $dq_{\ast}=\frac{e^{(x_{\ast}-1)H(a)}}{\mathbb{E}^{q}[e^{(x_{\ast}-1)H(a)}]}dq$, we have
\begin{equation}
\mathbb{E}^{q_{\ast}}[H(a)]=\frac{\mathbb{E}^{q}\left[e^{\theta_{\ast}+(x_{\ast}-1)H(a)}\right]}
{\mathbb{E}^{q}\left[e^{(x_{\ast}-1)H(a)}\right]}=\frac{1}{x_{\ast}}-\frac{\nu}{x},
\end{equation}
and $\mathbb{E}^{q_{\ast}}[\log\frac{dq_{\ast}}{dq}]=(x_{\ast}-1)\mathbb{E}^{q_{\ast}}[H(a)]-\log\mathbb{E}^{q}[e^{(x_{\ast}-1)H(a)}]$, 
which imply
\begin{align}
&\liminf_{t\rightarrow\infty}\frac{1}{t}\log\mathbb{P}\left(\frac{N_{t}}{t}\in B_{\epsilon}(x)\right)
\\
&\geq-\inf_{\hat{q}}\left\{x\mathbb{E}^{\hat{q}}[H(a)]+\nu-x+x\log\left(\frac{x}{x\mathbb{E}^{\hat{q}}[H(a)]+\nu}\right)
+x\mathbb{E}^{\hat{q}}\left[\log\frac{d\hat{q}}{dq}\right]\right\}\nonumber
\\
&\geq-\left\{x\mathbb{E}^{q_{\ast}}[H(a)]+\nu-x+x\log\left(\frac{x}{x\mathbb{E}^{q_{\ast}}[H(a)]+\nu}\right)
+x\mathbb{E}^{q_{\ast}}\left[\log\frac{dq_{\ast}}{dq}\right]\right\}\nonumber
\\
&=\theta_{\ast}x-\nu(x_{\ast}-1)=\sup_{\theta\in\mathbb{R}}\{\theta x-\Gamma(\theta)\}.\nonumber
\end{align}
\end{proof}

\begin{remark}
An alternative proof of the lower bound of the large deviation principle in Theorem \ref{LDP} is by checking the essential smoothness
condition in G\"{a}rtner-Ellis theorem and applying it to prove the lower bound directly. Nevertheless, our tilting approach
has the advantage of pin-pointing to the most likely path when a rare event occurs.
\end{remark}

\section{Risk Model with Marked Hawkes Claims Arrivals}\label{RiskModel}

We consider the following risk model for the surplus process $R_{t}$ of an insurance portfolio,
\begin{equation}
R_{t}=u+\rho t-\sum_{i=1}^{N_{t}}C_{i},
\end{equation}
where $u>0$ is the initial reserve, $\rho>0$ is the constant premium and the $C_{i}$'s are i.i.d. positive random variables
with the common distribution $\mu(dC)$. $C_{i}$ represents the claim size at the $i$th arrival time, 
these being independent of $N_{t}$, a marked Hawkes process.

For $u>0$, let
\begin{equation}
\tau_{u}=\inf\{t>0: R_{t}\leq 0\},
\end{equation}
and denote the infinite and finite horizon ruin probabilities by
\begin{equation}
\psi(u)=\mathbb{P}(\tau_{u}<\infty),\quad\psi(u,uz)=\mathbb{P}(\tau_{u}\leq uz),\quad u,z>0.
\end{equation}

We first consider the case when the claim sizes have light-tails, i.e. there exists some $\theta>0$ so that
$\int e^{\theta C}\mu(dC)<\infty$.

By the law of large numbers,
\begin{equation}
\lim_{t\rightarrow\infty}\frac{1}{t}\sum_{i=1}^{N_{t}}C_{i}=\frac{\mathbb{E}^{\mu}[C]\nu}{1-\mathbb{E}^{q}[H(a)]}.
\end{equation}
Therefore, to exclude the trivial case, we need to assume that
\begin{equation}\label{between}
\frac{\mathbb{E}^{\mu}[C]\nu}{1-\mathbb{E}^{q}[H(a)]}<\rho<\frac{\nu(x_{c}-1)}{\theta_{c}},
\end{equation}
where the critical values $\theta_{c}$ and $x_{c}=f(\theta_{c})$ satisfy
\begin{equation}\label{ftheta}
\begin{cases}
x_{c}=\int_{\mathbb{R}^{+}}\int_{\mathbb{X}}e^{\theta_{c}C+H(a)(x_{c}-1)}q(da)\mu(dC)
\\
1=\int_{\mathbb{R}^{+}}\int_{\mathbb{X}}H(a)e^{H(a)(x_{c}-1)+\theta_{c}C}q(da)\mu(dC)
\end{cases}.
\end{equation}

Following the proofs of large deviation results in Section \ref{LDPProof}, we have
\begin{equation}
\Gamma_{C}(\theta):=\lim_{t\rightarrow\infty}\frac{1}{t}\log\mathbb{E}\left[e^{\theta\sum_{i=1}^{N_{t}}C_{i}}\right]
=
\begin{cases}
\nu(x-1) &\text{if $\theta\in(-\infty,\theta_{c}]$}
\\
+\infty &\text{otherwise}
\end{cases},
\end{equation}
where $x$ is the minimal solution to the equation
\begin{equation}
x=\int_{\mathbb{R}^{+}}\int_{\mathbb{X}}e^{\theta C+(x-1)H(a)}q(da)\mu(dC).
\end{equation}

Before we proceed, let us quote a result from Glynn and Whitt \cite{Glynn}, which
will be used in our proof of Theorem \ref{InfiniteHorizon}.

\begin{theorem}[Glynn and Whitt \cite{Glynn}]\label{GlynnThm}
Let $S_{n}$ be random variables. $\tau_{u}=\inf\{n: S_{n}>u\}$ and $\psi(u)=\mathbb{P}(\tau_{u}<\infty)$. 
Assume that there exist $\gamma,\epsilon>0$ such that

(i) $\kappa_{n}(\theta)=\log\mathbb{E}[e^{\theta S_{n}}]$ is well defined and finite for $\gamma-\epsilon<\theta<\gamma+\epsilon$.

(ii) $\limsup_{n\rightarrow\infty}\mathbb{E}[e^{\theta(S_{n}-S_{n-1})}]<\infty$ for $-\epsilon<\theta<\epsilon$.

(iii) $\kappa(\theta)=\lim_{n\rightarrow\infty}\frac{1}{n}\kappa_{n}(\theta)$ exists and is finite for $\gamma-\epsilon<\theta<\gamma+\epsilon$.

(iv) $\kappa(\gamma)=0$ and $\kappa$ is differentiable at $\gamma$ with $0<\kappa'(\gamma)<\infty$.

Then, $\lim_{u\rightarrow\infty}\frac{1}{u}\log\psi(u)=-\gamma$.
\end{theorem}

\begin{remark}
We claim that $\Gamma_{C}(\theta)=\rho\theta$ has a unique positive solution $\theta^{\dagger}<\theta_{c}$. 
Let $G(\theta)=\Gamma_{C}(\theta)-\rho\theta$. Notice that $G(0)=0$, $G(\infty)=\infty$, and that $G$ is convex. 
We also have $G'(0)=\frac{\mathbb{E}^{\mu}[C]\nu}{1-\mathbb{E}^{q}[H(a)]}-\rho<0$ and $\Gamma_{C}(\theta_{c})-\rho\theta_{c}>0$ since 
we assume that $\rho<\frac{\nu(f(\theta_{c})-1)}{\theta_{c}}$. Therefore, there exists only one
solution $\theta^{\dagger}\in(0,\theta_{c})$
of $\Gamma_{C}(\theta^{\dagger})=\rho\theta^{\dagger}$.
\end{remark}

\begin{theorem}[Infinite Horizon]\label{InfiniteHorizon}
Assume all the assumptions in Theorem \ref{LDP} and in addition \eqref{between}, we have
$\lim_{u\rightarrow\infty}\frac{1}{u}\log\psi(u)=-\theta^{\dagger}$, where $\theta^{\dagger}\in(0,\theta_{c})$ 
is the unique positive solution of $\Gamma_{C}(\theta)=\rho\theta$.
\end{theorem}

\begin{proof}
Take $S_{t}=\sum_{i=1}^{N_{t}}C_{i}-\rho t$ and $\kappa_{t}(\theta)=\log\mathbb{E}[e^{\theta S_{t}}]$. 
Then $\lim_{t\rightarrow\infty}\frac{1}{t}\kappa_{t}(\theta)=\Gamma_{C}(\theta)-\rho\theta$. 
Consider $\{S_{nh}\}_{n\in\mathbb{N}}$. We have $\lim_{n\rightarrow\infty}\frac{1}{n}\kappa_{nh}(\theta)=h\Gamma_{C}(\theta)-h\rho\theta$. 
Checking the conditions in Theorem \ref{GlynnThm} and applying it, we get 
\begin{equation}
\lim_{u\rightarrow\infty}\frac{1}{u}\log\mathbb{P}\left(\sup_{n\in\mathbb{N}}S_{nh}>u\right)=-\theta^{\dagger}.
\end{equation}
Finally, notice that
\begin{equation}
\sup_{t\in\mathbb{R}^{+}}S_{t}\geq\sup_{n\in\mathbb{N}}S_{nh}\geq\sup_{t\in\mathbb{R}^{+}}S_{t}-\rho h.
\end{equation}
Hence, $\lim_{u\rightarrow\infty}\frac{1}{u}\log\psi(u)=-\theta^{\dagger}$.
\end{proof}

\begin{theorem}[Finite Horizon]
Under the same assumptions as in Theorem \ref{InfiniteHorizon}, we have
\begin{equation}
\lim_{u\rightarrow\infty}\frac{1}{u}\log\psi(u,uz)=-w(z),\quad\text{for any $z>0$}.
\end{equation}
Here
\begin{equation}
w(z)=
\begin{cases}
z\Lambda_{C}\left(\frac{1}{z}+\rho\right) &\text{if $0<z<\frac{1}{\Gamma'_{C}(\theta^{\dagger})-\rho}$}
\\
\theta^{\dagger} &\text{if $z\geq\frac{1}{\Gamma'_{C}(\theta^{\dagger})-\rho}$}
\end{cases},
\end{equation}
$\Lambda_{C}(x)=\sup_{\theta\in\mathbb{R}}\{\theta x-\Gamma_{C}(\theta)\}$
and $\theta^{\dagger}\in(0,\theta_{c})$ 
is the unique positive solution of $\Gamma_{C}(\theta)=\rho\theta$, as before.
\end{theorem}

\begin{proof}
The proof is similar to that in Stabile and Torrisi \cite{Stabile} and we omit it here.
\end{proof}

Next, we are interested to study the case when the claim sizes have heavy tails, i.e. $\int_{\mathbb{R}^{+}}e^{\theta C}\mu(dC)=+\infty$
for any $\theta>0$.

A distribution function $B$ is subexponential, i.e. $B\in\mathcal{S}$ if 
\begin{equation}
\lim_{x\rightarrow\infty}\frac{\mathbb{P}(C_{1}+C_{2}>x)}{\mathbb{P}(C_{1}>x)}=2,
\end{equation}
where $C_{1}$, $C_{2}$ are i.i.d. random variables with distribution function $B$. 
Let us denote $B(x):=\mathbb{P}(C_{1}\geq x)$
and let us assume that $\mathbb{E}[C_{1}]<\infty$ and define $B_{0}(x):=\frac{1}{\mathbb{E}[C]}\int_{0}^{x}\overline{B}(y)dy$,
where $\overline{F}(x)=1-F(x)$ is the complement of any distribution function $F(x)$.

Goldie and Resnick \cite{Goldie} showed that if $B\in\mathcal{S}$ and satisfies some smoothness
conditions, then $B$ belongs to the maximum domain of attraction of either the Frechet distribution
or the Gumbel distribution. In the former case, $\overline{B}$ is regularly varying,
i.e. $\overline{B}(x)=L(x)/x^{\alpha+1}$, for some $\alpha>0$ and we write
it as $\overline{B}\in\mathcal{R}(-\alpha-1)$, $\alpha>0$.

We assume that $B_{0}\in\mathcal{S}$ and either $\overline{B}\in\mathcal{R}(-\alpha-1)$ or $B\in\mathcal{G}$, i.e.
the maximum domain of attraction of Gumbel distribution, that is, 
there exist sequences $a_{n}>0$, $b_{n}\in\mathbb{R}$, such that
$\lim_{n\rightarrow\infty}n\overline{B}(a_{n}x+b_{n})=e^{-x}$, $x\in\mathbb{R}$.
$\mathcal{G}$ includes Weibull and lognormal distributions.

When the arrival process $N_{t}$ satisfies a large deviation result, the probability that it deviates away
from its mean is exponentially small, which is dominated by subexonential distributions. The results in Zhu \cite{ZhuVI}
for the asymptotics of ruin probabilities for risk processes with non-stationary, 
non-renewal arrivals and subexponential claims can be applied in the context of marked Hawkes arrivals.
We have the following infinite-horizon and finite-horizon ruin probability estimates when the claim sizes
are subexponential.

\begin{theorem}
Assume the net profit condition $\rho>\mathbb{E}[C_{1}]\frac{\nu}{1-\mathbb{E}^{q}[H(a)]}$.

(i) (Infinite-Horizon)
\begin{equation}
\lim_{u\rightarrow\infty}\frac{\psi(u)}{\overline{B}_{0}(u)}
=\frac{\nu\mathbb{E}[C_{1}]}{\rho(1-\mathbb{E}^{q}[H(a)])-\nu\mathbb{E}[C_{1}]}.
\end{equation}

(ii) (Finite-Horizon) For any $T>0$,
\begin{align}
&\lim_{u\rightarrow\infty}\frac{\psi(u,uz)}{\overline{B}_{0}(u)}
\\
&=
\begin{cases}
\frac{\nu\mathbb{E}[C_{1}]}{\rho(1-\mathbb{E}^{q}[H(a)])-\nu\mathbb{E}[C_{1}]}
\left[1-\left(1+\left(\frac{\rho(1-\mathbb{E}^{q}[H(a)])-\nu\mathbb{E}[C_{1}]}{\rho(1-\mathbb{E}^{q}[H(a)])}\right)
\frac{T}{\alpha}\right)^{-\alpha}\right]
&\text{if $\overline{B}\in\mathcal{R}(-\alpha-1)$}
\\
\frac{\nu\mathbb{E}[C_{1}]}{\rho(1-\mathbb{E}^{q}[H(a)])-\nu\mathbb{E}[C_{1}]}
\left[1-e^{-\frac{\rho(1-\mathbb{E}^{q}[H(a)])-\nu\mathbb{E}[C_{1}]}{\rho(1-\mathbb{E}^{q}[H(a)])}T}\right]&\text{if $B\in\mathcal{G}$}
\end{cases}.\nonumber
\end{align}
\end{theorem}

\section{Examples with Explicit Formulas}

In this section, we discuss two examples where an explicit formula exists.

Example \ref{EX1} is about the exponential asymptotics of the infinite-horizon ruin probability when $H(a)$ 
and the claim size $C$ are exponentially distributed.
Example \ref{EX2} gives an explicit expression for the rate function of the large deviation principle
when $H(a)$ is exponentially distributed.

\begin{example}\label{EX1}
Recall that $x$ is the minimal solution of
\begin{equation}
x=\int_{\mathbb{R}^{+}}\int_{\mathbb{X}}e^{\theta C+(x-1)H(a)}q(da)\mu(dC).
\end{equation}
Now, assume that $H(a)$ is exponentially distributed with parameter $\lambda>0$, then, we have
\begin{equation}
x=\mathbb{E}^{\mu}[e^{\theta C}]\frac{\lambda}{\lambda-(x-1)},
\end{equation}
which implies that
\begin{equation}
x=\frac{1}{2}\left\{\lambda+1-\sqrt{(\lambda+1)^{2}-4\lambda\mathbb{E}^{\mu}[e^{\theta C}]}\right\}.
\end{equation}
Now, assume that $C$ is exponentially distributed with parameter $\gamma>0$. Then, 
\begin{equation}
x=\frac{1}{2}\left\{\lambda+1-\sqrt{(\lambda+1)^{2}-4\lambda\frac{\gamma}{\gamma-\theta}}\right\}.
\end{equation}
The infinite horizon probability satisfies $\lim_{u\rightarrow\infty}\frac{1}{u}\log\psi(u)=-\theta^{\dagger}$,
where $\theta^{\dagger}$ satisfies
\begin{equation}
\rho\theta^{\dagger}
=\nu\left(\frac{1}{2}\left\{\lambda+1-\sqrt{(\lambda+1)^{2}-4\lambda\frac{\gamma}{\gamma-\theta^{\dagger}}}\right\}-1\right),
\end{equation}
which implies
\begin{equation}
\frac{2\rho\theta^{\dagger}}{\nu}+1-\lambda=-\sqrt{(\lambda+1)^{2}-\frac{4\lambda\gamma}{\gamma-\theta^{\dagger}}},
\end{equation}
and thus
\begin{equation}
\frac{\rho^{2}}{\nu^{2}}(\theta^{\dagger})^{2}+\frac{\rho\theta^{\dagger}}{\nu}(1-\lambda)
=\lambda-\frac{\lambda\gamma}{\gamma-\theta^{\dagger}}=\frac{-\lambda\theta^{\dagger}}{\gamma-\theta^{\dagger}}.
\end{equation}
Since we are looking for positive $\theta^{\dagger}$, we get the quadratic equation,
\begin{equation}
\rho^{2}(\theta^{\dagger})^{2}-(\rho^{2}\gamma-\rho\nu(1-\lambda))\theta^{\dagger}-(\rho\nu\gamma(1-\lambda)+\lambda\nu^{2})=0.
\end{equation}
Since $\rho>\frac{\mathbb{E}^{\mu}[C]\nu}{1-\mathbb{E}^{q}[H(a)]}=\frac{\nu\lambda}{\gamma(\lambda-1)}$, we have
$\rho\nu\gamma(1-\lambda)+\lambda\nu^{2}>0$. Therefore,
\begin{equation}
\theta^{\dagger}=\frac{(\rho^{2}\gamma-\rho\nu(1-\lambda))+\sqrt{(\rho^{2}\gamma-\rho\nu(1-\lambda))^{2}
+4\rho^{2}(\rho\nu\gamma(1-\lambda)+\lambda\nu^{2})}}{2\rho^{2}}.
\end{equation}
\end{example}

\begin{example}\label{EX2}
Now, let $H(a)$ be exponentially distributed with parameter $\lambda>0$. We want
an explicit expression for the rate function of the large deviation principle for $(N_{t}/t\in\cdot)$.
Notice that,
\begin{equation}
\Gamma(\theta)
=
\begin{cases}
\nu\left(\frac{1}{2}\left\{\lambda+1-\sqrt{(\lambda+1)^{2}-4\lambda e^{\theta}}\right\}-1\right)
&\text{for $\theta\leq\log\left(\frac{(\lambda+1)^{2}}{4\lambda}\right)$}
\\
+\infty &\text{otherwise}
\end{cases}.
\end{equation}
To get $I(x)=\sup_{\theta\in\mathbb{R}}\{\theta x-\Gamma(\theta)\}$, we optimize over $\theta$
and consider $x=\Gamma'(\theta)$. Evidently,
\begin{equation}
x+\frac{1}{2}\nu(-4\lambda)e^{\theta}\frac{1}{2\sqrt{(\lambda+1)^{2}-4\lambda e^{\theta}}}=0,
\end{equation}
which gives us
\begin{equation}
\theta=\log\left(\frac{-2x^{2}+x\sqrt{4x^{2}+\nu^{2}(\lambda+1)^{2}}}{\lambda\nu^{2}}\right),
\end{equation}
whence,
\begin{equation}
I(x)
=
\begin{cases}
x\log\left(\frac{-2x^{2}+x\sqrt{4x^{2}+\nu^{2}(\lambda+1)^{2}}}{\lambda\nu^{2}}\right)
\\
\qquad\qquad\qquad
-\nu\left(\frac{1}{2}\left\{\lambda+1-\frac{-2x+\sqrt{4x^{2}+\nu^{2}(\lambda+1)^{2}}}{\nu}\right\}-1\right)
&\text{if $x\geq 0$}
\\
+\infty &\text{otherwise}
\end{cases}.
\end{equation}
\end{example}

\section*{Acknowledgements}

The authors are both supported by NSF grant DMS-0904701, DARPA grant and MacCracken Fellowship at New York University.
The authors are grateful to the editor and the referees for a very careful reading of the manuscript
and also for the helpful suggestions.

\end{document}